\documentclass[notitlepage]{scrartcl}
\usepackage[shortlabels]{enumitem}
\usepackage{amsmath}
\usepackage{amssymb}
\usepackage{amsthm}
\usepackage{abstract}
\usepackage{xcolor}
\usepackage{comment}
\usepackage{mathtools}
\usepackage{graphicx}
\usepackage{caption}
\usepackage{subcaption}
\usepackage{float}
\usepackage{hyperref}
\usepackage{cleveref}
\hypersetup{
    colorlinks=true,
    linkcolor=blue,
    filecolor=magenta,      
    urlcolor=cyan
    }
\makeatletter
\newcommand*{\barfix}[2][.175ex]{%
  \mathpalette{\@barfix{#1}}{#2}%
}
\newcommand*{\@barfix}[3]{%
  \vbox{%
    \kern#1\relax
    \hbox{$#2#3\m@th$}%
  }%
}
\makeatother

\newtheorem{theorem}{Theorem}

\newtheorem{lemma}[theorem]{Lemma}
\newtheorem{proposition}[theorem]{Proposition}

\newcommand{\footremember}[2]{%
    \footnote{#2}
    \newcounter{#1}
    \setcounter{#1}{\value{footnote}}%
}
\newcommand{\footrecall}[1]{%
    \footnotemark[\value{#1}]%
} 

\usepackage[margin=1in]{geometry} 

\title{\vspace{-2em}Hitting time of connectedness in the random hypercube process}
\author{%
Sahar Diskin \footremember{alley}{School of Mathematical Sciences, Tel Aviv University, Tel Aviv 6997801, Israel. \\Emails: sahardiskin@mail.tau.ac.il, krivelev@tauex.tau.ac.il.}%
\and Michael Krivelevich \footrecall{alley}%
}

\begin{document}
\maketitle
\vspace{-2em}
\begin{abstract}
We present a short and self-contained proof of a classical result due to Bollob\'as (1990): in the random hypercube process, with high probability the hitting time of connectedness equals the hitting time of having minimum degree at least one.
\end{abstract}

The purpose of this expository note is to present a short and self-contained proof of the classical result of Bollob\'as \cite{B90} about the hitting time of connectedness in the random hypergraph process. 

Recall that the $d$-dimensional binary hypercube $Q^d$ is the graph whose vertex set is $V=\{0,1\}^d$, and for every $\bar{x},\bar{y}\in V$, $\{\bar{x},\bar{y}\}\in E(Q^d)$ if they differ in exactly one coordinate. Thus, $Q^d$ is a $d$-regular graph on $n\coloneqq 2^d$ vertices. For $p=p(d)\in [0,1]$, the random subgraph $Q^d_p\subseteq Q^d$ is obtained by retaining each edge of $Q^d$ independently and with probability $p$. 

In 1977, Burtin \cite{B77} showed that $p=\frac{1}{2}$ is the threshold for connectedness in $Q^d_p$: for a fixed value of $p$, if $p<\frac{1}{2}$ then \textbf{whp}\footnote{With high probability, that is, with probability tending to one as $d$ tends to infinity. Throughout this note, we assume $d$ is an asymptotic parameter tending to infinity.} $Q^d_p$ is disconnected, and if $p>\frac{1}{2}$ then \textbf{whp} $Q^d_p$ is connected. Erd\H{o}s and Spencer \cite{ES79} extended this result and showed that when $p=\frac{1}{2}$, the probability that $Q^d_p$ is connected tends to $e^{-1}$. In 1990, Bollob\'as \cite{B90} proved the aforementioned hitting time result for the random hypercube process, from which one can also derive an accurate expression for the probability of the connectedness of $Q^d_p$.

Let us briefly recall that the random hypercube process starts with $Q(0)$ being the empty graph on $\{0,1\}^d$. At each step $1\le i \le \frac{nd}{2}$, $Q(i)$ is obtained from $Q(i-1)$ by adding uniformly at random a new edge from $Q^d$. Note that $Q(t)$ can be seen as choosing uniformly at random $t$ edges of $Q^d$. The hitting time of a monotone increasing, non-empty graph property $\mathcal{P}$ is a random variable equal to the index $\tau$ for which $Q(\tau)\in \mathcal{P}$, but $Q(\tau-1)\notin\mathcal{P}$.

\begin{theorem}\label{th: main}[\cite{B90}]
Consider the random hypercube process. Let $\tau_D$ be the hitting time of minimum degree (at least) one and let $\tau_C$ be the hitting time for connectedness. Then \textbf{whp} $\tau_D=\tau_C$.    
\end{theorem}

In fact, we will mainly work in $Q^d_p$. We will prove the following proposition. 
\begin{proposition}\label{prop: main}
Let $\epsilon>0$ be a sufficiently small constant, and let $p\ge \frac{1}{2}-\epsilon$. Then \textbf{whp} there is a unique connected component in $Q^d_p$ on at least $(1-o(1))n$ vertices. Furthermore, \textbf{whp} all the other components in $Q^d_p$ (if there are any) are isolated vertices, and every two isolated vertices in $Q^d_p$ are at distance at least two in $Q^d$.
\end{proposition}
Note that if a graph satisfies the conclusions of \Cref{prop: main} and is disconnected, adding edges to it one by one, it will deterministically become connected the moment that its minimum degree is one. Indeed, we will show that if $p=\frac{1}{2}-\epsilon$ then \textbf{whp} $Q^d_p$ has isolated vertices, and then \Cref{th: main} will follow from \Cref{prop: main} by standard coupling.

Before proving \Cref{prop: main}, let us first collect some auxiliary (standard and/or easy) lemmas. The first one bounds from above the number of trees of a given size containing a given vertex in a graph of bounded degree.
\begin{lemma}\label{l: trees}
Let $G$ be a graph of maximum degree $d$, and let $k>0$ be an integer. For every $v\in V(G)$, let $t(v,k)$ be the number of trees in $G$ on $k$ vertices rooted at $v$. Then $t(v,k)\le (ed)^{k-1}$. 
\end{lemma}
We present its short proof, as given in \cite[Lemma 2]{BFM98}, for completeness.
\begin{proof}
Let $\mathcal{T}(v,k)$ be the family of trees in $G$ on $k$ vertices rooted at $v$. Then $t(v,k)=|\mathcal{T}(v,k)|$. Given a tree $T\in \mathcal{T}(v,k)$, label $v$ with $k$, and choose a labelling $f: V(T)\setminus\{v\}\rightarrow [k-1]$ for the remaining vertices of $T$. Note that every $T\in \mathcal{T}(v,k)$ is in $(k-1)!$ many such pairs $(T,f)$. Furthermore, each pair $(T,f)$ defines a unique spanning tree $T'$ of $K_k$, in which $(i,j)$ is an edge of $T'$ if and only if the vertices $x,y\in V(T)$ with $f(x)=i$, $f(y)=j$ are connected by an edge of $T$. Thus, all that remains is to estimate in how many ways a spanning tree $T'$ of $K_k$ can be obtained under such labellings. 

Fix (say) a BFS order on $T'$ starting from $k$, and upon reaching the vertex $\ell\in [k-1]$ for the first time, define $f^{-1}(\ell)$. As $f^{-1}(\ell)$ must be a neighbour in $G$ of the preimage of the (already embedded) father of $\ell$, there are at most $d$ ways to define $f^{-1}(\ell)$. Hence, using Cayley's formula to count the number of spanning trees $T'$ of $K_k$, we obtain that
\begin{align*}
\text{\# of pairs } (T,f)= t(v,k)\cdot (k-1)! \le d^{k-1}k^{k-2},    
\end{align*}
implying $t(v,k)\ge \frac{k^{k-2}}{(k-1)!}d^{k-1}$. Since $\frac{k^{k-2}}{(k-1)!}<e$ for every integer $k>0$, the lemma follows.    
\end{proof}

We will also need some control of the edge-expansion of sets in $Q^d$. To that end, the following two approximate versions of Harper's edge isoperimetric inequality \cite{H64} will suffice. As we will shortly see, their proof is fairly simple. Given $A,B\subseteq V(Q^d)$ with $A\cap B=\varnothing$, we denote by $e(A,B)$ the number of edges with one endpoint in $A$, and their other endpoint in $B$, and by $e(A)$ the number of edges spanned by $A$.
\begin{lemma}\label{l: harper for small sets}
For every $S\subseteq V(Q^d)$ we have that $e(S,S^C)\ge |S|(d-2\log_2|S|)$.
\end{lemma}
\begin{proof}
It suffices to show that for every $S\subseteq V(Q^d)$, if the minimum degree of $Q^d[S]$ is $\delta$, then $|S|\ge 2^{\delta}$. Indeed, since every graph of average degree $\bar{d}$ contains a subgraph of minimum degree at least $\bar{d}/2$, we then have that $|S|\ge 2^{\frac{e(S)}{|S|}}$, that is, $e(S)\le |S|\log_2|S|$. Since $Q^d$ is $d$-regular, we obtain $e(S,S^C)=d|S|-2e(S)\ge |S|(d-2\log_2|S|)$, as required.

To the task at hand, suppose the minimum degree of $Q^d[S]$ is $\delta$. Fix an arbitrary $v\in S$, and for $0\le i \le d$, let $S_i\subseteq S$ be the set of vertices whose distance (in $Q^d$) from $v$ is exactly $i$. Since $Q^d$ is bipartite, it does not contain cycles of odd length, and thus for every $i\in [0,d]$ we have that $e(S_i)=0$. Furthermore, all edges from $u\in S_i$ to $S$ have their other endpoints in $S_{i-1}\cup S_{i+1}$. Moreover, observe that $u\in S_i$ sends at most $i$ edges to $S_{i-1}$. Since the minimum degree of $Q^d[S]$ is $\delta$, it follows that for every $i\in [0, d-1]$, $e(S_i,S_{i+1})\ge (\delta-i)|S_i|$. As every $w\in S_{i+1}$ sends at most $i+1$ edges to $S_i$, we conclude that $|S_{i+1}|\ge \frac{\delta-i}{i+1}|S_i|$. Thus, by induction, it follows that $|S_i|\ge\binom{\delta}{i}$ for every $i\in [0,\delta]$. Altogether, we obtain $|S|\ge \sum_{i=0}^{\delta}|S_i|\ge \sum_{i=0}^{\delta}\binom{\delta}{i}=2^{\delta}$.
\end{proof}

The above bound will be too weak (or even vacuous) for large enough sets in $Q^d$. For such sets, the following `folklore' bound will suffice for us. Given a subgraph $H\subseteq Q^d$ and subsets $A,B\subseteq V(H)$ with $A\cap B=\varnothing$, we denote by $e_H(A, B)$ the number of edges in $H$ with one endpoint in $A$ and the other endpoint in $B$.
\begin{lemma}\label{l: harper for big sets}
For every $S\subseteq V(Q^d)$ with $|S|\le 2^{d-1}$ we have that $e(S,S^C)\ge |S|$.
\end{lemma}
\begin{proof}
We prove by induction on $d$. For $d=1$, the claim follows immediately. Suppose $d>1$. Let $Q_0$ be the subcube of $Q^d$ obtained by fixing the first coordinate to be zero, and let $Q_1$ be the subcube of $Q^d$ obtained by fixing the first coordinate to be one. Note that both $Q_0,Q_1$ are isomorphic to $Q^{d-1}$. Let $S_0\coloneqq S\cap V(Q_0), S_1\coloneqq S\cap V(Q_1)$, and note that $S=S_0\sqcup S_1$. Suppose without loss of generality that $|S_1|\ge |S_0|$. Then, by the induction hypothesis, since $|S_0|\le \frac{1}{2}|S|\le 2^{d-2}$, we have $e_{Q_0}(S_0, V(Q_0)\setminus S_0)\ge |S_0|$. Furthermore, by induction again, $e_{Q_1}(S_1, V(Q_1)\setminus S_1)\ge \min\left\{|S_1|, |V(Q_0)|-|S_1|\right\}\ge |S_0|$. Finally, each vertex of $S_1$ is incident to exactly one edge to $Q_0$, and these edges are disjoint, implying $e(S_1, V(Q_0)\setminus S)\ge |S_1|-|S_0|$. Altogether, we obtain that
\begin{align*}
    e(S,S^C)&\ge e_{Q_0}(S_0, V(Q_0)\setminus S_0)+e_{Q_1}(S_1, V(Q_1)\setminus S_1)+e(S_1, V(Q_0)\setminus S)\\&\ge |S_0|+|S_0|+|S_1|-|S_0|=|S|.
\end{align*}
\end{proof}

Turning to the proof of \Cref{prop: main}, let us first show that \textbf{whp} every two isolated vertices in $Q^d_p$ are at distance (in $Q^d$) at least two.
\begin{lemma}\label{l: isolated vertices are far}
\textbf{Whp}, every two isolated vertices in $Q^d_p$ are at distance at least two.
\end{lemma}
\begin{proof}
Fix an edge $\{u,v\}\in E(Q^d)$. The probability that $u$ and $v$ are isolated vertices in $Q^d_p$ is $(1-p)^{2d-1}=\frac{1}{2^{2d-1}}\left(1-2\epsilon\right)^{2d-1}\le n^{-3/2}$. There are $\frac{nd}{2}$ edges to consider. Thus, by the union bound, the probability that there are two isolated vertices in $Q^d_p$ at distance one in $Q^d$ is at most $\frac{nd}{2}n^{-3/2}=o(1)$.
\end{proof}

Let us also show that \textbf{whp} there are no components in $Q^d_p$ whose order is between $2$ and $n^{1/3}$.
\begin{lemma}\label{l: edges are far apart}
Let $\epsilon>0$ be a sufficiently small constant. Let $p\ge \frac{1}{2}-\epsilon$. Then, \textbf{whp} there are no components in $Q^d_p$ whose order is between $2$ and $n^{1/3}$.
\end{lemma}
\begin{proof}
Let $k\in [2, n^{1/3}]$. Let $\mathcal{A}_k$ be the event that there exists a component of order $k$ in $Q^d_p$. For $\mathcal{A}_k$ to occur, there should be a tree $T$ of order $k$ in $Q^d$ with all the edges of $Q^d$ between $V(T)$ and $V\setminus V(T)$ closed in $Q^d$. By Lemma \ref{l: harper for small sets}, there are at least $k(d-2\log_2k)$ such edges in $Q^d$. Thus, by Lemma \ref{l: trees} and by the union bound, we have that
\begin{align*}
    \mathbb{P}\left[\bigcup_{2\le k \le n^{1/3}}\mathcal{A}_k\right]&\le \sum_{k=2}^{n^{1/3}}n(ed)^{k-1}(1-p)^{k(d-2\log_2k)}\\
    &\le\sum_{k=2}^{d}n\left[ed\left(\frac{1}{2}+\epsilon\right)^{d-2\log_2k}\right]^k+\sum_{k=d+1}^{n^{1/3}}n\left[ed\left(\frac{1}{2}+\epsilon\right)^{\frac{d}{3}}\right]^k\\
    &\le d\cdot n^{-1/3}+n^{1/3} \cdot n^{-d/4}=o(1).
\end{align*}
\end{proof}

We are now ready to prove \Cref{prop: main}.
\begin{proof}[Proof of \Cref{prop: main}]
Let $p_2=\epsilon$, and let $p_1$ be such that $(1-p_1)(1-p_2)=1-p=\frac{1}{2}+\epsilon$. Note that $Q^d_{p_1}\cup Q^d_{p_2}$ has the same distribution as $Q^d_p$, and that $p_1\ge \frac{1}{2}-2\epsilon$.

The expected number of isolated vertices in $Q^d_{p_1}$ is $n(1-p_1)^d\le n^{9\epsilon}$. Thus, by Markov's inequality, \textbf{whp} there are at most $n^{10\epsilon}$ isolated vertices in $Q^d_{p_1}$. By Lemma \ref{l: isolated vertices are far}, in $Q^d_{p_1}$ \textbf{whp} there are no components of order between $2$ and $n^{1/3}$. Let $W$ be the set of vertices in components of size at least $n^{1/3}$ in $Q^d_{p_1}$, where their number $s$ is at most $n^{2/3}$. Let us show that after sprinkling with probability $p_2$, \textbf{whp} all the components in $W$ merge. Indeed, otherwise, there would have been a component respecting partition $A\sqcup B = W$, with no edges between $A$ and $B$ in $Q^d_{p_2}$. Suppose without loss of generality that $a=|A|\le |B|$. Then, by \Cref{l: harper for big sets}, $e(A, A^C)\ge |A|$. As there are at most $dn^{10\epsilon}$ edges in $Q^d$ touching $V\setminus W$, we derive $e(A,B)\ge |A|/2$. There are at most $\sum_{i=1}^{a/n^{1/3}}\binom{s}{i}\le s^{a/n^{1/3}}\le n^{a/n^{1/3}}$ ways to partition $W$ into two parts one of which is of order $a$. Thus, by the union bound the probability of this event is at most
\begin{align*}
    \sum_{a=n^{1/3}}^{n/2}n^{a/n^{1/3}}(1-\epsilon)^{a/2}\le \sum_{a=n^{1/3}}^{n/2}n^{a/n^{1/3}}e^{-\frac{\epsilon a}{2}}=o(1).
\end{align*}
Thus, \textbf{whp} all the components in $W$ merge after sprinkling with probability $p_2$.

Note that by Lemma \ref{l: isolated vertices are far}, \textbf{whp} every two isolated vertices in $Q^d_{p_1}$ (and in $Q^d_p$) are not connected by an edge of $Q^d$. Hence, adding any edge touching an isolated vertex connects it to a component whose order is at least $n^{1/3}$ in $Q^d_{p_1}$, and these components all merge \textbf{whp}. Therefore, \textbf{whp}, there exists a unique connected component in $Q^d_{p_1}$ whose order is at least $n-n^{10\epsilon}$, and all the other components are isolated vertices, whose distance in $Q^d$ is at least two.
\end{proof}

We conclude the note with the proof of \Cref{th: main}.
\begin{proof}[Proof of Theorem \ref{th: main}]
Note that $\tau_D\le \tau_C$ deterministically. Thus, it suffices to show that there exists $t$ such that \textbf{whp} $Q(t)$ has isolated vertices, and satisfies the properties of \Cref{prop: main}, that is, \textbf{whp} $Q(t)$ has a unique connected component of order $n-o(n)$ and all the other vertices of $Q(t)$ are isolated vertices whose distance between them (in $Q^d$) is at least two. Indeed, then adding any new edge in the process will either not affect the component structure, or will connect an isolated vertex to the unique connected component of order $n-o(n)$, and we will have that \textbf{whp} $\tau_D=\tau_C$.

Let $\mathcal{P}_1$ be the property of having no isolated vertices. Since $Q^d$ is bipartite, the number of isolated vertices in $Q^d_p$ stochastically dominates $Bin(n/2, (1-p)^d)$. Hence, setting $p=\frac{1}{2}-\epsilon$, the probability there are no isolated vertices is at most $\left(1-(1-p)^d\right)^{n/2}\le \exp\left\{-(\frac{1}{2}+\epsilon)^d\frac{n}{2}\right\}=o(1)$. Moreover, note that the properties in \Cref{prop: main} are increasing --- if $H\subseteq Q^d$ satisfies them, and $H\subset H'$, then $H'$ also satisfies them. Thus, using standard coupling between $Q^d_p$ and $Q(t)$, we obtain that there exists a $t<\tau_D$ such that \textbf{whp} $Q(t)$ has a unique connected component of order $n-o(n)$ and all its other vertices are isolated vertices whose distance in $Q^d$ is at least two. 
\end{proof}

\bibliographystyle{abbrv}
\bibliography{perc} 
\end{document}